\documentclass[12pt]{article}

\usepackage{diagbox}
\usepackage{mathtools}
\usepackage{bbm}
\usepackage{latexsym}
\usepackage{epsfig}
\usepackage{amsmath,amsthm,amssymb,enumerate}
\usepackage{float}

\newlength{\wdth}

\usepackage{xcolor}

\def\reals{\mathbb R}

\def\nn{\nonumber}
\def\a{\alpha} \def\b{\beta} \def\d{\delta} 
\def\e{\varepsilon} \def\f{\phi} \def\F{{\Phi}}  \def\g{\gamma}
\def\G{\Gamma}  
 \def\th{\theta}    \def\l{\lambda}
 \def\m{\mu} \def\n{\nu} \def\p{\pi}
\def\r{\rho}  \def\s{\sigma} 
\def\t{\tau} \def\om{\omega}

\newtheorem{theorem}{Theorem}
\newtheorem{lemma}[theorem]{Lemma}

\newtheorem{Remark}{Remark}

\newcommand{\proofend}{\hspace*{\fill}\mbox{$\Box$}}

\newcommand{\ooi}{(1+o(1))}

\newcommand{\ul}[1]{\mbox{\boldmath$#1$}}

\newcommand{\wh}[1]{\widehat{#1}}

\newcommand{\rdown}[1]{{\left\lfloor #1\right \rfloor}}

\newcommand{\brac}[1]{\left(#1\right)}

\newcommand{\bfrac}[2]{\left(\frac{#1}{#2}\right)}

\def\cE{{\cal E}}
\newcommand{\rai}{\rightarrow \infty}

\newcommand{\set}[1]{\left\{#1\right\}}

\def\seq{\subseteq}
\def\es{\emptyset}

\def\E{\mathbb{E}}

\def\Pr{\mathbb{P}}
\def\whp{{\bf w.h.p.}}

\newcommand{\ignore}[1]{}

\def\cA{{\mathcal A}}

\def\cE{{\mathcal E}}

\def\cW{{\mathcal W}}

\def\var{{\bf Var}}
\newcommand{\beq}[2]{\begin{equation}\label{#1}#2\end{equation}}

\def\nn{\nonumber}

\newcommand{\Q}[3]{\Pr_{#1,#2}^{(#3)}}

\begin{document}
\author{Colin Cooper\thanks{Research supported  at the University of Hamburg, by a  Mercator fellowship from DFG - Project 491453517}\\Department of Informatics\\
King's College\\
London WC2B 4BG\\England
\and
Alan Frieze\thanks{Research supported in part by NSF grant DMS1952285}\\Department of Mathematical Sciences\\Carnegie Mellon University\\Pittsburgh PA 15213\\USA}

\title{Random walks on edge colored random graphs}
\maketitle

\begin{abstract}
We consider random walks on edge coloured random graphs, where the colour of an edge reflects the cost of using it. In the simplest instance, the edges are coloured red or blue. Blue edges are free to use, whereas red edges incur a unit cost every time they are traversed.
\end{abstract}

\section{Introduction}
The cover time of a connected graph 
is the maximum over the start vertex  of the expected time  for a simple random walk to visit every vertex of the graph. There is a large literature on this subject  see for example \cite{AFill}, \cite{LPW},  \cite{Lovasz}, and including  \cite{CF1}--\cite{CFRHyper} which give precise estimates of cover time for various models of random graphs.

 We consider the following scenario: There is a  connected graph $G$ in which the edges are  colored red and blue. We study the cover time of such a graph when there is a bound on the number red edges that can be used. We can think of blue edges as free to use, whereas red edges represent toll roads, the bound being the maximum budget.
 The decision to use, or not to use, a toll road can be made in various ways, which we characterize as  congestion charging, flip walks, oblivious walks and smooth walks.
 In the simplest case, the {\bf oblivious model}, the random walk  uses edges of $G$ regardless of colour  until the budget is used up, after which it will be restricted to the blue sub-graph.

\paragraph{Models of red-blue regular graphs.}
Let $G$ be a connected $n$-vertex graph, and let $\cW=\cW(R,B)$ be a simple random walk on $G$ subject to the given constraints on the use of red and blue edges. We use $C_G=C_G(R,B)$ for the constrained cover time of $G$ by the walk $\cW(R,B)$, and use $\wh C_G$ for  the corresponding unconstrained cover time of $G$.

Let
\begin{equation}\label{rho-etc}
\s_{RB}=\frac{r+b-1}{r+b-2}, \qquad \s_B=\frac{b-1}{b-2}.
\end{equation}
One simple model is to take $G$ as the union of a random $n$-vertex $r$-regular graph with red edges and a random $n$-vertex $b$-regular graph with blue edges. For such graphs, w.h.p. the unconstrained cover time of $G$ (see \cite{CF1}) is
asymptotic to $\wh C_G=\s_{RB}n \log n=(r+b-1)/(r+b-2) n \log n$, whereas for $b \ge 3$, the cover time of a walk restricted to the blue edges is asymptotic to $C_G(B)=\s_{B}n \log n=(b-1)/(b-2) n \log n$.

Another simple model is to take $G$ as the union of a blue Hamilton cycle and a random $r$-regular graph with red edges.
This model, based on short cutting  a Hamilton cycle,
is one of the original Watts-Strogatz type small world models \cite{WS}.
For $r=1$ (the union of a Hamilton cycle and a random 1-factor) the diameter  was studied by Bollobás and Chung \cite{Boll1}.
Subsequently, the fact that 3-regular  graphs are Hamiltonian  w.h.p. was established by  Robinson and Wormald \cite{Wormy}.
Thus  almost all  3-regular graphs  have a decomposition into a blue Hamilton cycle and a random red 1-factor.

\paragraph{Smooth walks.}
The walk  alternates in phases between using edges of either colour (a red/blue phase) and edges of colour blue only (a blue phase). In a red/blue phase the walk is on  $G$, and in a blue phase the walk is on $G_B$, the subgraph induced by the blue edges. The length of the phases can vary, but to analyse the model, we  require that the length of each phase is at least $\om T_G\log n$ where $T_G$ is a mixing time given by \eqref{TG}, and $\om\to\infty$ with $n$. We call such  a walk {\em smooth}.
We say that such a random walk is {\em $\a$-constrained} if in total, it is only allowed to use a red edge at most $\g(\a) n\log n$ times, where $\g(\a)$ is given in Theorem \ref{Th1-S}(a) below.

The {\bf congestion pricing model} fits naturally as a special case of smooth walks.  Peak and off-peak periods alternate as in the smoothed model. The length $C,F$ (charged, free) of the  peak and off peak periods are fixed and in constant proportion. The walk  uses the red (charged) edges only in off-peak periods when they are free to use. In peak periods the walk uses only blue (free) edges. We assume  there are at most $n^{\th}$ distinct phases of equal length $C+F$, for some small $\th >0$ constant.
\begin{theorem}\label{Th1-S}{\sc (Smooth walks)}
Let $G$ be the union of a random $r$-regular graph with red edges and a random $b$-regular graph with blue edges where $r\ge 1$ and $b \ge 2$.

\vspace{-0.2in}
\begin{enumerate}[(a)]
\item If $b\geq 3$ and $\a<1$ constant, then w.h.p. the cover time of an  $\a$-constrained  smooth random walk on $G$ with budget $\g(\a)n\log n$ where $\g(\a)= \a\s_{RB}r/(r+b)$ is asymptotically equal to $C_G(R,B)=(\a\s_{RB}+(1-\a)\s_B)n\log n$.
\item
If $b=2$, $r \ge 1$,  and the blue edges span a Hamilton cycle then in the congestion pricing model with periods $C,F$,
the cover time of a random walk on $G$ is asymptotically equal to $C_G=(1+C/F)((r+1)/r) \, n \log n$.
\end{enumerate}
\end{theorem}

\paragraph{Flip walks.} In a flip walk there is no hard bound on the number of times we use a red edge, only a stochastic one. At any step, let $\r_R$ be the probability of transition down a given red edge adjacent to the current vertex and $\r_B$ the probability of transition down a given blue edge. Thus the probability the next transition is red is $r\r_R$ etc. Examples include the following. Let $\r_R=\a/r$, and $\r_B=(1-\a)/b$ in which case the probability of transition along a red edge at any step is $\a$, and a blue edge is $1-\a$.

\begin{theorem}\label{Th1-F} {\sc (Flip walks)}
\vspace{-0.2in}
\begin{enumerate}[(a)]
\item {\em Simplest case.}
If $r=1$, $b=2$, and $q=2\r_B=1-\r_R$, where $0<\r_R \le \r_B$,  the cover time of a flip walk on $G$ is asymptotically equal to $C_G\sim \th n \log n$ where
\[
\th= \frac 2{q(5-q+\sqrt{9-10q+q^2})}+ \frac 2{q(1-q+\sqrt{9-10q+q^2})}.
\]
This cover time is minimized at $2 n \log n$ when $\r_R=\r_B=1/3$, and all edges have the same transition probability.

\item If $r=1$, $b\geq 2$  and $0< \r_R \le \r_B$
then w.h.p. the cover time of a flip walk on $G$ is asymptotically equal to $C_G\sim (1/(1-f)) n \log n$.
Here $f$ is the smallest solution in $[0,1]$ to $F(z)=0$ where $F(z)$ is given by \eqref{F(z)}.
\end{enumerate}
\end{theorem}
Theorem \ref{Th1-F}(a) confirms that for the simplest case,  the cover time cannot be improved by biassing the
walk probabilities. The general solution $r\ge 2, b\ge 2$ seems difficult to obtain.

\paragraph{Oblivious walks.}
The simplest case.  The random walk  uses edges of $G$ regardless of colour  until the budget is used up, after which it is be restricted to the blue subgraph. In the case where $b=2$, $r \ge 1$ and the budget is less than the unconstrained cover time $C_G \sim (r+1)/r \, n \log n$, things  go seriously wrong for the walk; see Theorem \ref{Th1-O}(b)--(c) below.

\begin{theorem}\label{Th1-O} {\sc (Oblivious walks)} Let $\e>0$ constant.
\vspace{-0.2in}
\begin{enumerate}[(a)]
\item Suppose $b \ge 3$, and the budget is $(1-\e) \s_{RB} n \log n$.  
The cover time of an oblivious walk on $G$ is asymptotically equal to $(\a\s_{RB}+ (1-\a)\s_B) n \log n$ w.h.p., as given in \eqref{rho-etc} where $\a=\tfrac{1-\e}{r(r+b)}$.
\item Suppose that $b=2,r \ge 1$, and the blue edges span a Hamilton cycle. If  the
budget is at most $(1-\e)(r+1)/r n \log n$, 
the cover time of an oblivious walk on $G$ is asymptotic to $C_G(R,B)\sim n^2/2$, w.h.p.
\item If $b=2$ and the blue edges span a random 2-factor then w.h.p. an oblivious walk
with budget of  at most $(1-\e)(r+1)/r n \log n$ steps
will fail to cover $G$.
\end{enumerate}
\end{theorem}

\paragraph{Related models.}
The underlying graph of the random walk we consider varies structurally over time (temporally) in a predictable fashion; but within the constraint of a fixed budget.
Many authors have considered aspects of this, or related problems.
We mention a few below.

For a survey on temporal networks, see Holme \cite{Holme}.
Random walks whose  transition probability alters depending on structural condition have been studied extensively.
An early model by Kemperman \cite{Kemp}, the oscillating random walk, remains one of the simplest and most attractive.
Citing from \cite{Kemp}: Let $\{Y_n,n=0,1,...\}$ denote a stationary Markov chain taking values in $\reals^d$. As long as the process stays on the same side of a fixed
hyperplane $E$, it behaves as a ordinary random walk with jump measure $\mu$ or $\nu$, respectively.

As a general instance of an application of random walks with edge costs,
 Coppersmith,  Doyle,  Raghavan and Snir, \cite{CDRS}, use random
walks on graphs with positive real costs on all edges to study the design and analysis of randomized on-line algorithms.
Two recent papers  by Majumdar et al., \cite{Maju1}, \cite{Maju2}, which study variable charges for walk transitions, have some bearing on the model we consider. In \cite{Maju1}, \cite{Maju2}, a one-dimensional walker’s position $X_t$ at discrete time $t$ is a
positive random variable evolving according to $X_t=X_{t-1}+\eta_t$.
Transitions are priced according to their length $\eta$, there being one price for short distances and another for
longer ones. This reflects how taxi fares are calculated in
congested cities\footnote{\cite{Maju2} The municipality prescribes a threshold speed $\eta_c$
derived from statistical analyses of local traffic patterns. If
the taxi surpasses $\eta_c$, the meter tallies the fare based on the
distance covered, while a slower pace results in time-based
fare computation. This approach ensures that
drivers are compensated even when they face prolonged periods of slow progress.}.
The authors compute the average and variance of the distance covered in $n$ steps when the total
budget $C$ is fixed.

\paragraph{Proof outline.}

The main tool in our proofs will be the {\em First Time Visit Lemma},
a technique which will allow us to make very precise estimates of cover time in certain classes of graphs. The technique was  introduced in \cite{CF1}, and subsequently  refined  in \cite{CFGiant}, and in \cite{CFRHyper} where various
technical conditions were removed.
In outline, the approach is to bound the cover time above and below using
\eqref{frat} of Lemma \ref{MainLemma}.  Much of the work is in constructing the lower bound on the cover time.
On regular graphs the walks we consider have uniform stationary distribution, and are rapidly mixing.
The unvisit probability \eqref{frat} has a parameter $p_v\sim \pi_v/R_v$ (see \eqref{pv}) where $R_v$ is the expected number of returns
to vertex $v$ during the mixing time.  On regular graphs a simple random walk has a uniform stationary distribution, and is rapidly mixing
on random $r$-regular graphs provided the vertex degree $r$ is at least three.

\section{Background material}

\paragraph{Notation.}
We use $A_n \sim B_n$ to denote $A_n=\ooi B_n$ and thus $\lim_{n \rai} A_n/B_n=1$.
We use $\om$ to denote a quantity which tends to infinity with $n$ more slowly than any other variables in the given expression. The expression $ f(n) \ll g(n)$ indicates $f(n)=o(g(n))$. The inequality $A \lesssim B$ which stands for $A \le \ooi B$ is used to unclutter notation in some places. A sequence of events $\cE_n$ occurs {\em with high probability},
(w.h.p.), if
$\lim_{n\to\infty}\Pr(\cE_n)=1$.

\paragraph{Chernoff Bounds}
We use the following versions of the Chernoff bounds, where $Bin(n,p)$ denotes a binomial random variable.
\begin{align}
\Pr(Bin(n,p)\geq (1+\e)np)&\leq e^{-\e^2np/3}\qquad\text{for $0\leq \e\leq 1$}.\label{Chern1}\\
\Pr(Bin(n,p)\leq (1-\e)np)&\leq e^{-\e^2np/2}\qquad\text{for $0\leq \e\leq 1$}.\label{Chern2}\\
\Pr(Bin(n,p)\geq \a np)&\leq \bfrac{e}{\a}^{\a np}\qquad\text{for $\a>0$}.\label{Chern3}
\end{align}

\subsection{First Time Visit Lemma }

 Let $G$ denote a fixed connected graph,  and let $u$ be some arbitrary vertex from which a walk $\cW_{u}$  is started. Let $\cW_{u}(t)$ be the vertex reached at step $t$, let $P$ be the matrix of transition probabilities of the walk, and let $P_{u}^{(t)}(v)=\Pr(\cW_{u}(t)=v)$. Let $\pi$ be the steady state distribution of the random walk $\cW_{u}$. For an unbiased ergodic random walk on a graph $G$ with $m=m(G)$ edges, $\pi_v=\frac{d(v)}{2m}$, where $d(v)$ denotes the degree of $v$ in $G$.

We denote the steady state probability of a simple random walk on $G=(V,E)$ by $\p_G(v)$ for $v\in V$. We let $\Q{u}{G}{t}(v)$ denote the probability that a simple random walk on $G$ started at $u$ is at $v$ after $t$ steps. We then define a mixing time $T_G$ by
\begin{equation}\label{TG}
T_G=\min\set{\t:\max_{u,v}\set{\frac{|\Q{u}{G}{t}(v)-\p(v)|}{\p(v)}}=o(1)\text{ for }t\geq \t}.
\end{equation}
For $t\geq 0$, let $\cA_t(v)$ be the event that $\cW_u$ does not visit $v$ in steps $T_G,T_G+1,\ldots,t$. The vertex $u$ will have to be implicit in this definition. We will use the following version, it can be found in \cite{CFGiant}:
\begin{lemma}\label{MainLemma}
Let $R_v\geq 1$ is the expected number of visits by $\cW_v$ to $v$ in the time interval $[0,T_G]$. Suppose that $T_G\pi_v=o(1)$ and $R_v=O(1)$. Let
\begin{equation} \label{pv}
p_v=\frac{\pi_v}{R_v(1+O(T_G\pi_v))}.
\end{equation}
Then for all $t\geq T_G$,
\begin{equation}
\label{frat}
\Pr(\cA_t(v))=\frac{(1+O(T_G\pi_v))}{(1+p_v)^{t}} +O(T_G^2\p_v e^{-\m t/2}),
\end{equation}
where $\m=\Omega(1/T_G)$.
\end{lemma}

\subsection{Properties of random regular graphs}\label{Preg}
The use of \eqref{frat} will require us to verify some facts about random regular graphs. These will by and large be taken from \cite{CF1}. Let
\begin{equation}\label{sigs}
\s=\rdown{\log^{1/2} n}.
\end{equation}
Say a cycle $C$ is {\em small} if $|C|\leq \s$. A vertex of $G$ is locally tree-like if $v$ is at distance at least $\s$ from the nearest small cycle.
\begin{theorem}
Let $s\geq 3$ be a constant and let $G$ be chosen uniformly from the set ${\cal G}_s$ of $s$-regular graphs with vertex set $[n]$. Then w.h.p.
\begin{description}
\item [P1.] $G$ is connected.
\item[P2.] $T_G=O(\log n)$.
\item [P3.] If $v$ is locally tree-like then $R_v=\frac{s-1}{s-2}+o(\s^{-1})$.
\item[P4.] There are at most $s^{3\s}$ non locally-tree-like vertices and if $v$ is not locally-tree-like then $R_v\leq 3$.
\item [P5] The second eigenvalue $\l$ of $G$ is positive and  $\l\le 1-\th$ for some positive constant $\th$.
\end{description}
\end{theorem}

\section{Proof of Theorem \ref{Th1-S}}
\subsection{Proof of Theorem \ref{Th1-S}(a)}\label{Proof1a}
In this section $G$ is the union of a random $r$-regular graph with  red edges, and a random $b$-regular graph with blue edges, where $r\ge 1$ and $b \ge 3$. We prove that if the budget is $\g(\a)n\log n$, where $\g(\a)= \a\s_{RB} r/(r+b)$,
then w.h.p. the cover time of an  $\a$-constrained  smooth random walk on $G$
is asymptotically equal to $C_G=(\a\s_{RB}+(1-\a)\s_B)\,n\log n$.

It is shown in \cite{CF1} that  w.h.p. the unconstrained cover time of $G$ is $\wh C_G\sim \s_{RB} \,n\log n$. Thus without any restriction on the choice of the next edge,  a walk of length $\wh C_G$ would cross a red edge $\sim \frac{r}{r+b} \s_{RB} \,n \log n$ times. It follows that, for $\a<1$,  w.h.p. we will not be able to complete the walk within the budget without having to avoid red edges a significant number of times.

If the phases are less than a mixing time  in length ($o(T_G)$), so that the walk is not smooth, there are some technical difficulties that we have not yet managed to overcome. 
\begin{Remark}\label{rem1}
There does not seem to be a compelling reason for us to only consider smooth walks. If we allow the walk to switch arbitrarily between allowing red edges and not allowing them, then because $b \ge 3$ there are no problems with the walk mixing. This is because the blue subgraph is an expander and the steady state probabilities are $1/n$ regardless. The problem arises because the $R_v$  depend heavily on how we switch between the two  graphs.
\end{Remark}

\subsubsection{Upper bound}
Suppose now that we start a walk at vertex $u$ and that we alternate between red/blue phases and blue phases. Take $t_0=0$, and suppose  we change phase at times $t_1,t_2,\ldots,t_i,\ldots,t_m$ where $m=o(n^2)$. Assume that whenever $i$ is even we start a red/blue phase  at $t_i$  and switch to a blue phase when $i$ is odd. At each step in a red/blue phase there is an $r/(r+b)$ chance of using a red edge. We also know that this walk will finish at least as early as one that only uses blue edges. This follows from Lemma \ref{MainLemma} and the fact that $s/(s-1)$ decreases as $s$ increases.
 We will assume that $t_i-t_{i-1}\geq \om \log^2n$, where $\om \log^2n\gg T_G$, for $1\leq i<m$. It follows  that
\beq{rb0}{
t_{RB}=\sum_{i\ge 0}(t_{2i+1}-t_{2i})\leq \brac{1-\frac{1}{\log^{1/2}n}}\a\s_{RB} n\log n.
}
Applying the Chernoff bound \eqref{Chern1}, we see that if $Z_i$ denotes the number of times a red edge is used in $[t_{2i},t_{2i+1}]$, then where $\e=1/\log^{1/2}n$,
\[
\Pr(Z_R\geq \g(\a)n\log n)\leq \sum_{i=0}^m\Pr\brac{Z_i\geq \frac{(1+\e)(t_{2i+1}-t_{2i})r}{r+b}}\leq
\sum_{i=0}^m\exp\set{-\frac{\e^2r\om \log^2n}{r+b}}=o(1).
\]
Let $T_G(u)$ be the time taken to visit every vertex of $G$ by the random walk $\cW_u$. Let $U_t$ be the number of vertices of $G$ which have not been visited by $\cW_u$ at step $t$.
We note the following:
\begin{eqnarray}
C_u=\E T_G(u)&=& \sum_{t > 0} \Pr(T_G(u) \ge t), \label{ETG} \\
\label{TG-}
\Pr(T_G(u) > t)&=&\Pr(U_t>0)\le \min\{1,\E U_t\}.
\end{eqnarray}
It follows from (\ref{ETG}), (\ref{TG-}) that for all $t$
\begin{equation}
\label{shed}
C_u \le t+ \sum_{s \ge t} \E U_s=t+\sum_{v\in V}\sum_{s \ge t}\Pr(\ul A_s(v)).
\end{equation}
In both phases the random walk is on a regular graph and so $\p_v=1/n$ throughout.

We are stopping the RB walk at $t^*\sim\s_{RB} n \log n$ when the budget runs out. At that time the walk has made $t_{RB}\sim\a \s_{RB} n \log n$ RB steps, and $t_B\sim (1-\a)\s_{B} n \log n$ B steps, consisting of  $m$ phases each of B and RB up to $t^*$. This is followed by a final 'blue forever' phase of length $t'_B=t-t^*$.

So,
\beq{11}{
\Pr(\ul A_{t}(v))=\brac{1+O\bfrac{\log n}{n}}^m \exp\set{-\frac{\p_v}{\s_{RB}}
\sum_{i\geq 0}(t_{2i+1}-t_{2i})-\frac{\p_v}{\s_B} \sum_{i\geq 1}(t_{2i}-t_{2i-1})}.
}
Now assuming that $t=O(n\log n)$ and because each phase has length at least $\om\log^2n$ we have $m=o(n/\log n)$, and so we have
\begin{align}
\Pr(\ul A_{t}(v))&\lesssim \exp\set{-\brac{\a+\frac1{\s_B}\brac{\frac{t}{n\log n}-\a\s_{RB}}}\log n}\nn\\
&=\exp\set{-\frac{t}{\s_B n}+\a\s_{RB}\brac{\frac{1}{\s_B}-\frac{1}{\s_{RB}}}\log n}.\label{Asv}
\end{align}

For non locally tree-like vertices, using P4, we can write
\beq{Asw}{
\Pr(\ul A_{t}(v))\lesssim e^{-t/3n}.
}

Define a time $t_0=(1+\e)\t_C$, by
\[
t_0= (1+\e) \brac{\a\s_{RB}+(1-\a)\s_B n\log n}.
\]
We note this value of $t_0$, (and $t_1$ below), is unrelated to our previous use of notation $t_i$ for a smoothed period, but is consistent with \cite{CF1}.

Thus if $\e=\log^{-1/2}n$
going back to \eqref{shed} and using P4, \eqref{Asv}, we see that for some absolute constants $\g_1,\g_2>0$,
\begin{align}
C_u&\leq t_0+\sum_{v\in [n]}\sum_{\t\geq 0}\Pr(\ul A_{t_0+\t})\nn\\
&\leq t_0+ne^{-\g_1\log^{1/2}n}\sum_{\t\geq 0}e^{-\g_2\t\log n}+(r+b)^{3\log\log n}\sum_{\t\geq t_0}e^{-\t/3n}\nn\\
&=t_0+o(n).\label{upper1}
\end{align}
This confirms the upper bound in Theorem \ref{Th1-S}(a).

\subsubsection{Lower bound} \label{LowerBd}
For the lower bound we let $t_1=\tfrac{(1-\e)\a(r+b)}{(b-2)(r+b-1)}n\log n$ and argue that $C_u\geq t_1$ w.h.p. Proceeding as in \cite{CF1}, for any vertex $u$, we can find a set of vertices $S$ such that at time $t_1$, the probability the set $S$ is covered by the walk $\cW_u$ tends to zero. Hence $T_G(u) > t_1$ \whp\ which together with \eqref{upper1} completes the proof of Theorem \ref{Th1-S}(a).

We construct $S$ as follows. Let $\s$ be given by \eqref{sigs}. Let $S \seq [n]$ be some maximal set of locally tree-like vertices all of which are at least distance $2\s+1$ apart. Thus $|S| \ge (n-r^{3\s})/r^{(2\s+1)}$.

Let $S_1$ denote the subset of $S$ which has not been visited by $\cW_u$ after step $t_1$.
\begin{align}
\E |S_1|&\sim\sum_{v \in S} \exp\set{-\frac{t_1}{\s_Bn}+\brac{1+O\bfrac{\log n}{n^{1/2}}}\frac{\a(r+b)\log n}{r}\brac{\frac{1}{\s_{RB}}-\frac{1}{\s_B}}}\nn\\
&\geq (n-r^{3\s})r^{-(2\s+1)}\exp\set{-\frac{\e\a(r+b)\log n}{(b-2)(r+b-1)}}\to\infty.\label{Stoinf}
\end{align}
 Let $Y_{v,t}$ be the indicator for the event that $\cW_u$ has not visited vertex $v$ at time $t$. Let $Z=\{v,w\} \subset S$. We will show below that
\begin{equation}\label{poi}
\E (Y_{v,t_1}Y_{w,t_1})\sim \exp\set{-2\brac{\frac{t_1}{\s_Bn}+\brac{1+O\bfrac{\log n}{n^{1/2}}}\frac{\a(r+b)\log n}{r}\brac{\frac{1}{\s_{RB}}-\frac{1}{\s_B}}}}.
\end{equation}
Thus
\begin{equation}\label{zzz}
\E (Y_{v,t_1}Y_{w,t_1})=(1+o(1))\E (Y_{v,t_1}) \E( Y_{w,t_1}).
\end{equation}
It follows from (\ref{Stoinf}) and (\ref{zzz}), that
$$\Pr(S(t_1)\neq 0)\geq \frac{(\E |S(t_1)|)^2}{\E (|S(t_1)|^2)}
=\frac{1}{\frac{{\bf E} |S_{t_1}|(|S_{t_1}|-1)}{({\bf E}
|S(t_1)|)^2}+(\E|S_{t_1}|)^{-1}}
=1-o(1).$$
Going back to \eqref{11}, we see that with these values of $\p_Z,R_Z$, this completes the proof of Theorem \ref{Th1-S}(a), apart from \eqref{poi} which we give next.

\paragraph{\bf Proof of (\ref{poi}).}
Let $\G$ be obtained from $G$ by merging $v,w$ into a single node $Z$. This node has degree $2r$ and every other node has degree $r$.

There is a natural measure-preserving mapping from the set of walks in $G$ which start at $u$ and do not visit $v$ or $w$, to the corresponding set of walks in $\G$ which do not visit $Z$. Thus the probability that $\cW_u$ does not visit $v$ or $w$ in the first $t$ steps is equal to the probability that a random walk $\wh{\cW}_u$ in $\G$ which also starts at $u$ does not visit $Z$ in the first $t$ steps.

That $\p_Z=\frac{2}{n}$ is clear. We also have
\[
R_v\leq R_Z\leq R_v+\sum_{t=2\s+1}^T(\p_w+\l^t)=R_v+O(T\l^{2\s}).
\]
The sum here being a bound on the probability that $\cW_v$ is at $w$ before the end of the mixing time. The value $\l$ is the second eigenvalue in {\bf P5}.
\proofend

\subsection{Congestion model: Theorem \ref{Th1-S}(b)}
Let $t_1=(1+\e)(1+C/F)\s_{RB}n\log n$ where $\e>0$ is arbitrary. The amount of time spent in a free phase is $(1+\e)\s_{RB}n\log n$ and so w.h.p. $G$ will be covered.

For the lower bound, let $t_2=(1-\e)(1+C/F)\s_{RB}n\log n$. We defer the proof that $G$ will not be covered until Section \ref{conj1}.

\section{Proof of Theorem \ref{Th1-F}}
\paragraph{Stationary distribution and mixing time.}
Let $d=r+b$ be the total degree of a vertex in the red/blue coloured graph $G$. Assume $d \ge 3$
and $G$ satisfies the conditions of Section \ref{Preg}.

In the flip model, $\r_R$ (resp. $\r_B$) is the probability of transition over a given red edge (resp. blue edge).
This corresponds to a weighted graph in which $w(e_R)=\r_R$, $w(e_B)=\r_B$ for red and blue edges respectively,
all vertices $v$ have weight $w(v)=1$ and $w(G)=n$. Thus $\pi_{RB}(v)=1/n$, and $G$ has uniform stationary
distribution. For a given set $S \seq V$, the conductance (bottleneck ratio) $\F_{RB}(S)$ is
\[
\F_{RB}(S)= \frac{Q(S,S^c)}{\pi(S)}=\frac{\sum_{v \in S,  w \in S^c} \pi_vP_{RB}(v,w)}{\sum_{v \in S}\pi_v}
=\frac{\sum_{v \in S, w \in S^c} P_{RB}(v,w)}{|S|}.
\]
Comparing with a simple random walk on $G$,
\[
\frac 1{|S|}{\sum_{v \in S, w \in S^c} P_{RB}(v,w)}\ge \min(\r_R,\r_B) \frac{E(S:S^c)}{d|S|}=\min(\r_R,\r_B) \F_{SRW}(S),
\]
where $E(S:S^c)$ is the number of edges between $S$ and $S^c$ and $ \F_{SRW}$ refers to the conductance of a simple random walk.

As $G$ satisfies {\bf P5}, the conductance $\F^*_{SRW}$ of a simple random walk on $G$ is  constant for some $0<c<1$.
Assuming $\min(\r_R,\r_B)$ is also a positive constant, the conductance $\F^*_{PB}$ of the flip walk is
constant, i.e.,
\[
\F^*_{PB}=\min_{\pi(S)\le 1/2}\F_{RB}\F(S) \ge \min(\r_R,\r_B) \F^*_{SRW}.
\]
It follows  that the flip walk on $G$ is rapidly mixing and satisfies \eqref{TG} as required.

\paragraph{Return probability.}
The proof of Theorem \ref{Th1-F} will follow from the proof in \cite{CF1}, once we establish $R_v$ for locally-tree-like vertices. To do this we argue as follows. The initial formulation is for the general case $r \ge 1, b \ge 2$. We  then prove existence of solutions for $r=1, b \ge 2$, and state the exact value of $R_v$ for the special case $r=1,b=2$. The journey to the answer is longer than we might expect.

Consider an infinite $d$-regular tree $T_d$, where $r+b=d$ and in which each vertex has $r$ red and $b$ blue edges.
For discussion we regard $T_d$ as rooted at a designated vertex $v$,  with all edges  nominally oriented outward from $v$. With this orientation, a transition away from $v$ (resp. $u$) at a vertex $u$, is a transition $uw$ down an out-edge of $u$. Relative to this orientation, a vertex other than $v$ is red if its in-edge is red, and blue if its in-edge is blue. With the exception of $v$ this induces an alternating pattern of red and blue vertices in $T_d$, in which e.g., a red vertex has $(r-1)$ red out-edges, and $b$ blue out-edges.

Recall that, in the flip model, $\r_R$ (resp. $\r_B$) is the probability of transition over a given red edge (resp. blue edge). Let $f$ be the probability of a first return to $v$ in $T_d$, then the expected number of returns to $v$ is  $R_v(T_d)=1/(1-f)$. Let $\f_{v,R}$ (resp. $\f_{v,B}$) be the probability of a return to $v$ given the particle exits $v$ on a red (resp. blue) edge.
Thus
\[
f= r\r_R \,\f_{v,R}+b \r_B\, \f_{v,B}.
\]
Let $u$ be a red vertex, and $uw$ a red edge pointing away from $u$ and thus away from $v$. Define $\f_{RR}=\f_{RR}^{(u)}(uw)$ by
\begin{align*}
\f_{RR}=& \Pr( \text{ the walk returns to }u \mid u \text{ is red and the walk exits away from } u \text{ on a red edge}),
\end{align*}
and define $\f_{RB}, \f_{BB}, \f_{BR}$ similarly. Note that $\f_{RR}=\f_{BR}=\f_{v,R}$. We denote this common value by $\f_R$.
For a red vertex $u$, let $\psi_R(u)$ be the probability the walk moves away from $v$ at $u$, and subsequently returns to $u$. Thus
with $P(u,w)$  the colour dependent transition probability over the edge $uw$,  $\f(uw)$ the colour dependent return probability, and $N^+_R(u)$ the red out-neighbours of $u$ etc.,
\[
\psi_R(u)=\sum_{w\in N^+(u)} P(u,w) \f(uw)=\sum_{w \in N^+_R(u)}\r_R\f_{RR}+\sum_{w \in N^+_B(u)}\r_B\f_{RB}.
\]
If $u$ is red, there are $r-1$ red out-neighbours, and $b$ blue out-neighbours, so
\[
\psi_R=(r-1)\r_R \f_R+ b \r_B \f_B.
\]
Suppose the transition is away from $u$ over a red edge $uw$ to a red out-neighbour $w$, then the walk moves back to $u$ with probability $\r_R$ and so
\begin{equation}\label{geom}
\f_{RR}^{(u)}(uw)=\r_R+\psi_R \r_R+\cdots+(\psi_R)^j \r_R+\cdots= \frac{\r_R}{1-\psi_R}.
\end{equation}
Here the term $(\psi_R)^j \r_R$ is the probability that the walk returns to $u$ after making $j$ returns to $w$ that do not include a visit to $u$.

For a walk arriving at $u$, three things can happen. It moves directly back towards $v$ (probability $\r_R >0$), moves away from $u$ and returns to $u$ later (probability $\psi_R$),
or moves away from $u$ and never returns (probability $\xi_R$). Thus
\begin{equation}\label{xi-etc}
\r_R+\psi_R+\xi_R=1,
\end{equation}
and $\psi_R \le 1-\r_R<1$, and similarly for $\psi_B<1$. This justifies  division by $1-\psi_R$ in \eqref{geom}.

We obtain
\begin{align}
\psi_R=& \frac{(r-1) \r_R^2}{1-\psi_R}+ \frac{b \r_B^2}{1-\psi_B}, \label{eq1}\\
\psi_B=& \frac{r \r_R^2}{1-\psi_R}+ \frac{(b-1) \r_B^2}{1-\psi_B},\label{eq2}\\
f=&\frac{r \r_R^2}{1-\psi_R}+ \frac{b \r_B^2}{1-\psi_B}.\label{eq3}
\end{align}
Do \eqref{eq1} and \eqref{eq2} always have  solutions  $\psi_R, \psi_B$ in $[0,1]$ and if more than one solution, which to choose?
This is discussed next. Assume for now that solutions exists. Then for a given solution $\psi_B$, using \eqref{eq2} and \eqref{eq3} we see that
\begin{equation}\label{f-psi-B}
f= \psi_B+ \frac{\r_B^2}{1-\psi_B}.
\end{equation}
As $1/R_v=1-f+o(1)$ in $G$, we obtain the required solution to the cover time of $G$ as
\begin{equation}\label{Cov-G}
C_G \sim \frac{1}{1-f} \;n \log n .
\end{equation}

\paragraph{General solution in the case $r=1$.}
As $r=1$ we have $\r_R+b\r_B=1$. Assume $b \ge 2$ and $\r_R \le \r_B$ so that $0 \le \r_R \le 1/(b+1)$.
Use \eqref{eq1} to get
\[
\frac{1}{1-\psi_R}=\frac{1-\psi_B}{1-\psi_B-b\r_B^2},
\]
and  substitute this in \eqref{eq2} giving
\[
\psi_B=\frac{\r_R^2 (1-\psi_B)}{1-\psi_B-b\r_B^2}+ \frac{(b-1)\r_B^2}{1-\psi_B}.
\]
Put $q=b\r_B$ and $\r_R=1-q$  where $b/(b+1) \le q \le 1$, and replace $\psi_B$ by $z$ in the above expression.
We look for solutions to  $F(z)=0$  where
\begin{equation}\label{F(z)}
F(z)= z-\frac{(1-q)^2 (1-z)}{1-z-q^2/b}- \frac{(b-1)q^2}{b^2(1-z)}.
\end{equation}
If $\r_R=1/(b+1)$ then all edges are equivalent in a $b+1$ regular tree and the roots $z^*$ are  $1/(b+1), b/(b+1)$.
If $q=1$ (no useable red edges) then the problem reverts to a $b$-regular tree, in which case the roots
$z^*$ of $F(z)$ are $1/b, (b-1)/b$.

 If $b=2,q=1$, $z=1/2$ is the only root to $F(z)=0$. Ignoring this case which corresponds to a random walk on the line, we will assume that $q<1$ and $b\geq 2$. Let $z_0=1-q^2/b$. In the interval $I_0=[0,z_0]$, $F(0)<0$, $F(z_0^-)=-\infty$. It can be checked that $F(z)$ is concave on $I_0$. Indeed, $z$ and $-\frac{(b-1)q^2}{b^2(1-z)}$ are concave on $[0,1)$ and we can write $-\frac{1-z}{1-z-q^2/b}=-1-\frac{q^2/b}{1-z-q^2/b}$.

In the interval $I_1=[z_0,1]$, $F(z_0^+)=\infty$, $F(1)=-\infty$ and $F(z)$ is monotone decreasing. So there are between one and three roots  in $[0,1]$ depending on what happens in $I_0$. We will establish that if $b\ge 2,\, q<1$, then $F(1/2)>0$ which implies there are two roots in $I_0$. We have
\[
F(1/2)=\frac 12 - \frac{b(1-q)^2}{b-2q^2}- \frac{2(b-1)q^2}{b^2}.
\]
It can be checked that $\partial F/\partial b>0$ for any fixed $q$, so as $F(1/2)$ is monotone increasing with $b$,
and $b=2$ has the largest range of $q$, (i.e., $2/3 \le q<1$), we check this case. After some simplification we obtain that if $b=2$,
\[
 F(1/2)= (1-q) \brac{\frac{1+q}{2}-\frac{1}{1+q}}.
\]
So $ F(1/2,b=2) >0$ boils down to $(1+q)^2 >2$. As $q\ge 2/3= b/(b+1)$, and $(5/3)^2>2$, it is indeed the case
that $F(1/2)>0$ as required.

\paragraph{The simplest case: Theorem \ref{Th1-F}(a).}
We solve $F(z)=0$, where $F(z)$ is given by \eqref{F(z)}, in the simplest case; namely $r=1$, and $b=2$. Thus $\r_R+2\r_B=1$. Put $q=1-\r_R=2\r_B$,
and $w=1-z$ to give, after some rearrangement,  $F(w)=N(w)/D(w)$. Here the numerator $N(w)$ is
\[
N(w)=w^3-\frac q2 (4-q)w^2+\frac 34 q^2w-\frac 18 q^4.
\]
It can be checked that $w=q/2$ is a factor of $N(w)$ and so
\[
N(w)=(w-q/2)(w^2 +w (q/2) (q-3)+ q^3/4).
\]
The roots of $F(w)=0$ are
$ q/2, \;\; (q/4) [(3-q) \pm \sqrt{9-10q+q^2}]$. We  prove  that the smallest root $z \in [0,1]$ is indeed the correct one.
We use $\r_B+ \psi_B+\xi_B=1$  from \eqref{xi-etc}. As $\r_B=q/2$ and $\psi_B=z=1-w$ we obtain $\xi_B=w-q/2$.

The root $w^-=(q/4) [(3-q) - \sqrt{9-10q+q^2}]$ implies $\xi_B <0$ for $q<1$. As $\xi_B$ is a probability, this is impossible.
The root $w=q/2$  (or $w^-(q=1)$) imply that $\xi_B=0$. This corresponds to a recurrent system, an answer we do not accept (see below).
Thus the only feasible root is $w^+=(q/4) [(3-q) + \sqrt{9-10q+q^2}]$ giving $\psi_B=z=1-w^+$, this being the smallest root of $F(z)=0$.

We briefly sketch why the system is not recurrent. For a random walk on the half-line $\{0,1,2,...\}$ starting from vertex 1, let $q,p$ be the probability of moving left and right respectively at any vertex $i >0$. Then assuming $q \le p$, the probability of absorption at the origin 0, is $\pi_0=q/p$. Thus for our system
\[
\pi_0 \le \max \set{ \frac{\r_R}{1-\r_R}, \frac{\r_B}{1-\r_B}}= \max \set{ \frac{\r_R}{2\r_B}, \frac{\r_B}{\r_B+\r_R}}.
\]
As we assume $0< \r_R \le \r_B$ it follows that $\pi_0<1$, $\xi>0$ and the system is not recurrent.

Refer to \eqref{f-psi-B} and \eqref{Cov-G}. With $\psi_B=1-w$ and $\r_B=q/2$.
 the value of $\th =1/(1-f)$  is
\[
\th= \frac 12 \brac{ \frac 1{w+q/2} + \frac 1{w-q/2}}.
\]
The largest value of $w$ among the roots, and hence the smallest value of $z=\psi_B$,
gives
\begin{equation}\label{theta-q}
\th= \frac 2{q(5-q+\sqrt{9-10q+q^2})}+ \frac 2{q(1-q+\sqrt{9-10q+q^2})}.
\end{equation}
The minimum of $\th(q)$ is at $q=2/3$, which follows by checking the derivative.
This corresponds to $\r_R=\r_B=1/3$.  The value
$q=2/3$   gives $\th=2$,  and hence
value of $C_G \sim 2 n \log n$ in \eqref{Cov-G}.
This confirms  the w.h.p. cover time of a random 3-regular graph from \cite{CF1} as the minimum cover time for this model.
The value of $\th$ is
plotted in Figure \ref{Fig1}.

\vspace{-0.1in}

\begin{figure}[H]
	\centering
	\includegraphics[scale=.65]{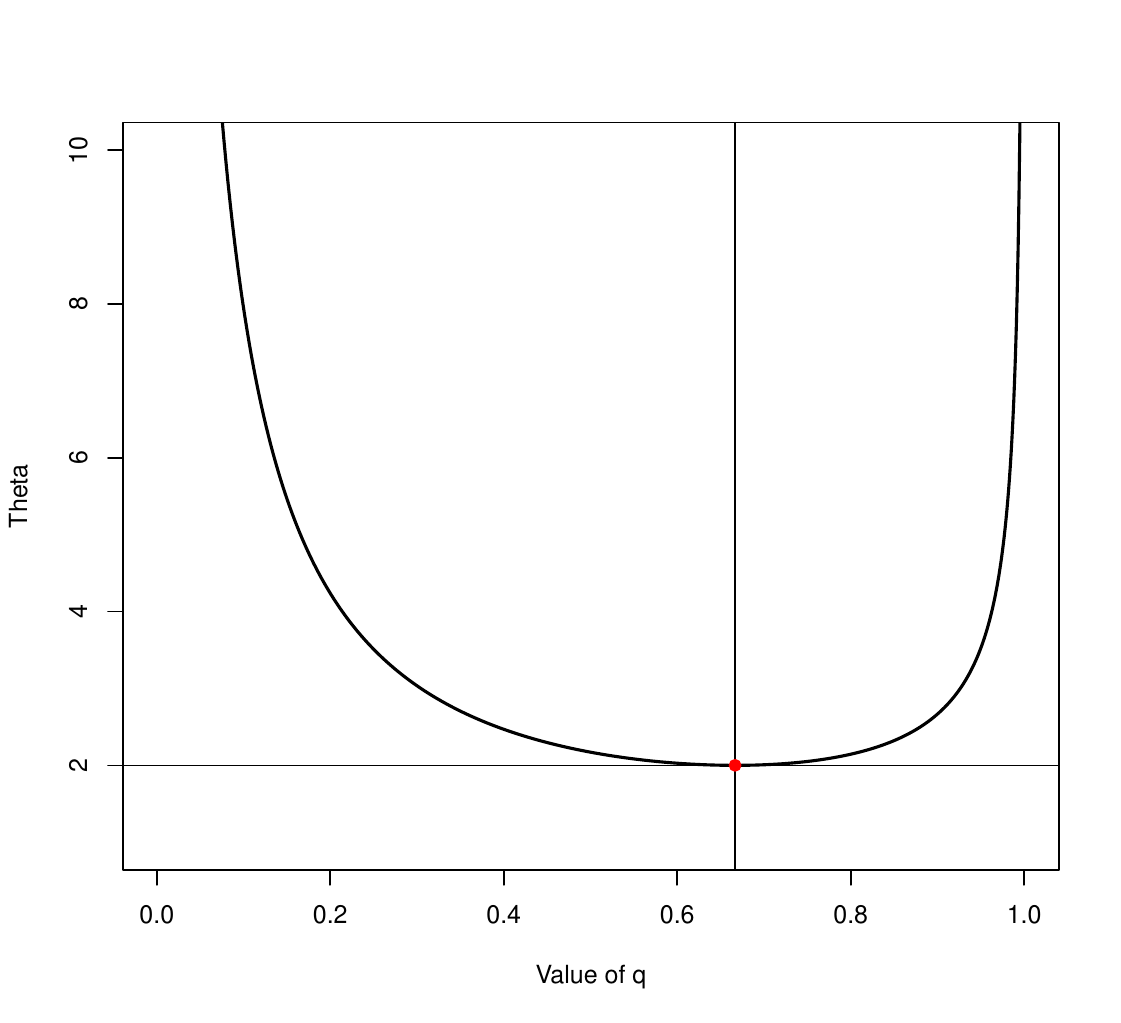}
	\caption{
Plot of  $\theta=\theta(q)$ from \eqref{theta-q} as $q$ varies from 0 to 1. The minimum value of $\theta=2$ is at $q=2/3$, giving a cover time of $2 n \log n$ for random 3-regular graphs. Values of $q$
for which $\r_R < \r_B$ are to the right of the vertical line at $q=2/3$. The plots diverge  at $q=0,1$.
 } \label{Fig1}
 \end{figure}

\section{The oblivious model: Proof of Theorem \ref{Th1-O}}
\subsection{Proof of Theorem \ref{Th1-O}(a).}
This follows directly from Theorem \ref{Th1-S}(a).
\subsection{Proof of Theorem \ref{Th1-O}(b).}\label{HAM}
We prove that if $b=2$, and the blue edges span a Hamilton cycle  $H$, then
 the cover time of an oblivious walk on $G$ with budget $t_1=(1-\e)(r+1)/r n \log n$  is asymptotic to $n^2/2$.

Divide $H$ into $\ell=n^{\nu}$ consecutive intervals $I_1,I_2,\ldots,I_\ell$ of length $L=n^{1-\nu}$, where  $0<\nu \le \e/3$. Fix one such interval $I$.  Let $S_0$ be the locally tree-like vertices in $I$ and then choose $S$ by greedily choosing a subset of $S_0$ where each pair of vertices are at distance $2\s+1$ apart. Note that $|S|\geq n^{1-\n-o(1)}$.

Let $S_1=S_1(t_1)$ denote the subset of $S$ which has not been visited by the walk  $\cW$ at step $t_1=(1-\e)(d-1)/(d-2) n \log n$, where $d=r+b$ as usual, and $b=2$.
Let $X_v$ be the indicator that vertex $v \in S$ is unvisited at step $t_1$.
Thus applying Lemma \ref{MainLemma} as in Section \ref{Proof1a}, we obtain
\begin{align*}
\E |S_1|= &\sum_{v \in S} X_v=
\sum_{v \in S}\exp\set{-(1+\eta)\frac{R_v}{n}\;(1-\e)\frac{d-1}{d-2} n \log n}\\
= &\frac{n^{1-\nu}}{d^{3\s}}\; n^{-1+\e+o(1)}= n^{\d},
\end{align*}
where $\d=\e-\nu+o(1)$. Here $\eta=o(1)$ combines the errors arising from estimating $R_v=(d-2)/(d-1) (1+o(1))$, the
convergence via \eqref{TG} of  $\cW$ to $\pi_v=1/n$, and the approximation of $p_v=\pi_v/R_v (1+O(T\pi_v)$
in Lemma \ref{MainLemma}.

We next prove that $|S_1(t_1)|>0$ by adapting the approach used in Section \ref{LowerBd}.
When calculating the variance of $|S_1|$ by contraction of pairs $u,v \in S_1$,
 the main error $\nu$ comes from the change in $R_v$  due to the probability that
 the  walk travels between the pair of contracted vertices $u,v$ in the mixing time.
  From \cite{CF1} this is
\begin{equation}\label{beta}
\b= \sum_{t=2\s}^{T_G}(\pi_v+\l^t)=o(1/\log n)
\end{equation}
assuming $\l=1-\th$ for some $\th>0$ constant, and $\s$ is given by \eqref{sigs}.

 For $X=\sum X_v$,
\begin{align*}
\E X(X-1)=&\sum_{u,v \in S} \E X_uX_v\\
=&\sum_{u,v \in S}\exp\set{-(1+\eta)\frac{2 }{n}R_{(u,v)}(1+\b)
\;(1-\e)\frac{d-1}{d-2} n \log n}\\
=&n^{2\d}n^{O(\b/n)}=(\E |S_1|)^2\brac{1+ O\bfrac{1}{n}},
\end{align*}
where $\b$ is defined in \eqref{beta}.  So,
\begin{align*}
\var |S_1| = &\E X(X-1)+ \E |S_1|-(\E |S_1|)^2\\
= &\E|S_1|\brac{1+ O\bfrac{ \E |S_1|}{n}}\\
=&\E|S_1|(1+o(1)).
\end{align*}
Using the Chebychev Inequality we deduce that
\[
\Pr(S_1=\es)<\Pr(|S_1-\E S_1| \ge \E S_1 /2) \le \frac{5}{\E S_1}=O\bfrac{1}{n^\d}.
\]
Let $\nu = \e/3$. The expected number of segments with $S_1=\es$ is at most
\[
 O\bfrac{\ell}{n^\d}=  O\bfrac{n^\nu}{n^{\e-\nu +o(1)}}=\frac{1}{n^{\nu+o(1)}}.
\]
We conclude that w.h.p. each segment $I$ has at least one unvisited  vertex in $S_1(I)$ when the budget runs out at $t_1$. Thus we have to effectively walk  around  the cycle to reach them all. This requires $\sim n^2/2$ time w.h.p. Let $C_{\sc OBV}( \e)$ be the cover time of an oblivious walk with budget at most $t_1=(1-\e)C(G)$, then
\[
C_{\sc OBV}(\e)\approx C(C_n) \sim \frac{n^2}{2}.
\]

\subsection{Proof of Theorem \ref{Th1-O}(c).}
We first observe that the space ${\cal F}(2,r)$ of graphs formed by the union of a random 2-factor and a random $r$-regular graph is contiguous to  the space ${\cal G}(r+2)$ of random $(r+2)$-regular graphs, see for example Wormald \cite{Wor}. So, the red/blue walk has mixing time $O(\log n)$ w.h.p. (we apply contiguity to the property of having a small second eigenvalue as in P5).

We next prove a basic lemma on random 2-factors.
\begin{lemma}\label{2fac}
W.h.p. a random 2-factor contains at least 2 cycles of length at least $cn /(\log^2 n)$ for some absolute constant $c>0$.
\end{lemma}
\begin{proof}
We begin by arguing that the number of cycles in a random 2-factor is $O(\log n)$ w.h.p. It follows that w.h.p. there is at least one cycle of length $\Omega(n/\log n)$, among which we choose the longest. We then argue that w.h.p. the maximum  length of this cycle is at most $\ell=n-n/\log n$.  As at least $n/\log n$ vertices remain after removing this cycle it follows that there is another cycle of length at least $\Omega(n/\log^2n)$.

\paragraph{Number of cycles.} Let $S_n$ denote the set of permutations of $[n]$, and let $S' \seq S_n$ be the subset with minimum cycle length 3. We can obtain a random 2-factor $F$ by choosing a random permutation from $S'$. This is not done uniformly, but we choose permutation $\p$ with probability proportional to  $\tfrac{1}{2^\ell}$ where $\ell=\ell(\p)$ is the number of cycles in $\p$. Having chosen $\p$ we obtain our 2-factor by ignoring orientation. A 2-factor  $F$ with $\ell=\ell(F)$ cycles arises from $2^\ell$ distinct permutations and so a 2-factor $F$ is chosen with probability $2^{\ell(F)}\cdot\frac{1}{2^{\ell(F)}}\cdot\frac{1}{M}$, where $M=\sum_{\p \in S'} \frac{1}{2^{\ell(\p)}}$.  Now the proof of Proposition 2(a) of Karp \cite{Kpatch} implies that
\beq{noc}{
\Pr(\ell(\p)\geq a\log n)\leq \Pr(Bin(a\log n,1/2)\leq \log_2n)\leq n^{-a/3}\qquad\text{for }a\geq 2.
}
So, $\Pr(\ell(\p)\geq 2\log n)\leq n^{-2/3}$. Because $|S'|/|S_n|=c\sim e^{-3/2}$ this also holds for $S'$. Summing over $S'$,
\[
M = \sum_{\ell(\p)< 2 \log_2 n} \frac{1}{2^{\ell(\p)}}+\sum_{\ell(\p)\ge  2 \log_2 n} \frac{1}{2^{\ell(\p)}} \geq \frac{c n!(1-O(n^{-2/3}))}{n^2}.
\]
And\[
\Pr(\ell(F)\geq 3\log_2n)=\sum_{m\geq 3\log_2n}
\frac{|\{\pi\in S':\ell(\pi)=m\}|}{2^{m}M}=O(n^{-1}).
\]
\paragraph{Longest cycle.}
We use the configuration model of Bollob\'as \cite{Boll}. In this model the expected number of cycles of length at least $\ell\geq \ell_0=n-n/\log n$ is
\begin{align*}
\sum_{k=\ell_0}^n\binom{n}{k}\frac{(k-1)!}{2}\cdot 2^k\cdot \frac{(2n-2k)!}{(n-k)!2^{n-k}}\cdot\frac{n!2^n}{(2n)!}&=
\sum_{k=\ell_0}^n \frac{2^{2k}}{2k} \frac{n! n!}{(2n)!} {2n-2k \choose n-k}\\
 =& O\bfrac{1}{\sqrt n}+O(\sqrt n)\sum_{k=\ell_0}^{n-1} \frac{1}{\sqrt{n-k}}  \frac 1k\\
 = & O\bfrac{1}{\sqrt n}+O (\sqrt{\log n})\; \log \bfrac{n}{n-n/\log n}\\
&=O\bfrac{1}{\sqrt{ \log n}}.
\end{align*}
\end{proof}
Given Lemma \ref{2fac} we proceed as follows: let $C_1,\,C_2$ be two of the cycles of length at least $\Omega(n/\log^2 n)$, as promised by Lemma \ref{2fac}. As  in Section \ref{HAM}, let $L=n^{1-\nu}$ where $\nu>0$. Partition each of the cycles $C_1,\,C_2$ into intervals of length $L$.  We know that w.h.p. at time {$t_1=(1-\e)(d-1)/(d-2) n \log n$, (where $d=r+b$) there will be unvisited vertices in $C_1,\,C_2$ and the result follows.

\subsection{Finishing the proof of Theorem \ref{Th1-S}(b). The lower bound}\label{conj1}
For the lower bound, let $t_2=(1-\e)(1+C/F)\s_{RB}n\log n$.
The length $C,F$  of the  charged and free periods are fixed  and in constant proportion.
The $t_2$ steps are divided into $\t=n^{\th}$ phases of equal length $C+F= n^{1-\th+o(1)}$.
The total amount of time spent in free phases is $(1-\e)\s_{RB}n\log n$.


In phase $1\leq i\leq \t$, during a charged period $C$ the walk will w.h.p. cover an interval $J_i$ of the Hamilton cycle of length $n^{(1-\th)/2+o(1)}$.
Let $I_1,I_2,\ldots,I_\ell$ be  intervals of length $L=n^{1-\nu}$, where $\nu=\e/3$  as in Section \ref{HAM}. In total, the intervals $J_1,J_2,\ldots,J_{\t}$ cover at most  $n^{\th} \,n^{(1+\th)/2 +o(1)}$ vertices.

We require $n^{\th+(1+\th)/2 +o(1)} \ll L= n^{1-\e/3}$, which is satisfied by $\th =(1-\e)/4$.
With this value of $\th$, w.h.p. each interval $I_j$ $j=1,...,\ell$ will contain a vertex unvisited during the free phases. This completes the proof of Theorem \ref{Th1-S}(b).

\end{document}